\documentclass[final,onefignum,onetabnum]{siamart190516}

\usepackage{mathtools,amssymb,dsfont}
\usepackage{nicematrix}
\usepackage{algorithm,algorithmicx,algpseudocode}
\usepackage{xcolor}
\usepackage{enumitem}
\usepackage[utf8]{inputenc}
\usepackage[T1]{fontenc}
\usepackage{lmodern}

\newcommand{\QEDA}{\hfill\ensuremath{\blacksquare}}%

\newsiamremark{remark}{Remark}
\newsiamremark{hypothesis}{Hypothesis}
\crefname{hypothesis}{Hypothesis}{Hypotheses}
\newsiamremark{ex}{Example}
\newsiamthm{claim}{Claim}
\newsiamthm{cor}{Corollary}

\headers{UEC of Real-Valued Diagonalizable Linear Ensembles}{W.~Miao, G.~Cheng, and J.-S.~Li}

\title{On Uniform Ensemble Controllability of Diagonalizable Linear Ensemble Systems
}
\author{Wei~Miao\thanks{Department of Electrical and Systems Engineering, Washington University, Saint Louis, MO 63130 (\email{weimiao@wustl.edu}, \email{gong.cheng@wustl.edu}, \email{jsli@wustl.edu}).}
  \and Gong~Cheng\footnotemark[2]
  \and Jr-Shin~Li\footnotemark[2]}

\usepackage{easyReview,todonotes}
\usepackage{cite,bookmark}
\usepackage{hyphenat}

\DeclareMathOperator{\diag}{diag}
\DeclareMathOperator{\rank}{rank}

\makeatletter
\renewcommand*\env@matrix[1][\arraystretch]{%
  \edef\arraystretch{#1}%
  \hskip -\arraycolsep
  \let\@ifnextchar\new@ifnextchar
  \array{*\c@MaxMatrixCols c}}
\makeatother

\begin{document}
\maketitle

\begin{abstract} 
  In this paper, we study uniform ensemble controllability (UEC) of linear ensemble systems defined in an infinite-dimensional space through finite-dimensional settings. Specifically, with the help of the Stone-Weierstrass theorem for modules, we provide an algebraic framework for examining UEC of linear ensemble systems with diagonalizable drift vector fields through checking the controllability of finite-dimensional subsystems in the ensemble. The new framework renders a novel concept of ensemble controllability matrix, which rank-condition serves as a sufficient and necessary condition for UEC of linear ensembles. We provide several examples demonstrating that the proposed approach well-encompasses existing results and analyzes UEC of linear ensembles not addressed by literature.
\end{abstract}


%

\section{Introduction}

In recent years, the ensemble control problem, which studies the collective behavior of population systems, has drawn much attention due to its broad applicability in diverse scientific areas. Notable examples include exciting a population of nuclear spins on the order of Avogadro's number in nuclear magnetic resonance (NMR) spectroscopy and imaging \cite{glaser1998unitary,li2011optimal}, spiking population of neurons to alleviate brain disorders such as Parkinson's disease \cite{brown2004phase,li2013control,kafashan2015optimal}, manipulating a group of robots under model perturbation \cite{becker2012approximate}, and creating synchronization patterns in a network of coupled oscillators \cite{rosenblum2004controlling,Li_NatureComm16}.

The distinctive characteristic of the ensemble control problem is that the control input can be applied only on the population level. Namely, the control signal is broadcast to all systems in the ensemble. This non-standard and under-actuated scheme arises since the number of systems in practical ensembles can be exceedingly large so that applying feedback control on each system is not possible. Therefore, the analysis of the ensemble control problem is full of new challenges and is far beyond the scope of classical control theory.

During the past decade, a rich amount of work has been developed, showing that various specialized techniques such as polynomial approximation \cite{li2011ensemble}, separating points \cite{li2019separating}, representation theory \cite{chen2019structure},complex functional analysis \cite{helmke2014uniform,schonlein2016controllability,dirr2018uniform}, statistical moment-based approaches \cite{zeng2016moment,zeng2017sampled}, and convex-geometric approaches \cite{miao2020convexgeometric} are nontrivially connected to analyzing ensemble controllability and ensemble observability. Apart from investigating fundamental properties of ensemble systems, customized approaches are proposed to design feasible and optimal control laws for linear \cite{Li_ACC12_SVD,miao2020convexgeometric,Li_ACC20_Iterative_Projection,tie2017explicit}, bilinear \cite{Li_SICON17,Li_Automatica18}, as well as specific types of nonlinear ensemble systems \cite{Li_TAC13}.

Although much of the work has been done, the discussion on ensemble controllability of linear ensemble systems is not finished yet. In \cite{li2019separating}, a transparent view on UEC of  real-valued linear ensemble systems under real-valued control inputs (R-ensembles) is provided by connecting UEC of an ensemble system to controllability of each subsystem in the ensemble. A sufficient and necessary condition on UEC of R-ensembles is given based on the separating properties of the control vector field. However, \cite{li2019separating} only focuses on the R-ensembles with the eigenvalues of their drift vector field to be all real (RR-ensembles). One promising way to examine UEC of R-ensembles with complex eigenvalues in the drift vector field (RC-ensembles) is to lift the dynamics of RC-ensembles into the complex domain, and then, use the results in \cite{dirr2018uniform} to analyze UEC of the corresponding complex-valued ensemble systems. Leveraging techniques from complex functional analysis, several necessary, as well as sufficient conditions on UEC of linear ensemble systems under complex-valued control inputs (C-ensembles) are presented in \cite{dirr2018uniform}. Nevertheless, \cite{dirr2018uniform} considers mostly the linear ensembles with their drift vector fields having non-overlapping spectra. Furthermore, since realistic systems are all steered by real-valued control inputs, the conditions in \cite{dirr2018uniform} appear to be restrictive and demanding for checking UEC of practical RC-ensembles. Aside from \cite{li2019separating} and \cite{dirr2018uniform}, a thorough investigation placed directly on UEC of RC-ensembles is missing in the literature. 

The paper is organized as follows. In Section~\ref{sec: Preliminaries}, we briefly summarize our previous work in \cite{li2019separating} and provide preliminary knowledge of the ensemble control problem. In Section~\ref{sec: UEC of Linear Ensembles through Stone-Weierstrass Theorem for Modules}, leveraging Stone-Weierstrass theorem for modules, we propose a new algebraic framework to examine UEC of linear ensembles through checking controllability of subsystems in the ensemble. We specifically propose a new concept of ensemble controllability matrix, which rank-condition can be used as a sufficient and necessary condition for examining UEC of linear ensembles. The novel framework presented in this work well-encompasses the results in \cite{li2019separating}, also renders a systematic approach to check UEC of linear ensembles that are not addressed by existing literature.

\section{Preliminaries}
\label{sec: Preliminaries}

In this section, we briefly review our previous work in \cite{li2019separating} with an emphasis on the definition of linear ensemble systems and ensemble controllability. Throughout this paper, we denote $[a:b] = \{a, a+1, \ldots , b-1, b\}$ and $ [\pm a: \pm b] = [-b: -a] \cup [b: a]$ for integers $a, b$ satisfying $a<b$. We also denote $\mathbb{N} =\{0, 1, 2, \ldots \}$ as the set of natural numbers. Given two metrizable spaces $S$ and $E$, we denote $C(S, E)$ as the space of $E$-valued continuous functions over $S$.


\begin{definition}[Ensemble Controllability]
  \label{def: ensemble controllability}
  Let $M$ be a manifold, $K$ be a Hausdorff space, and $\mathcal{F}(K)$ be a space of $M$-valued functions defined over $K$. Consider an ensemble of systems parameterized by a parameter $\beta \in K$ given by
  \[
    \frac{\mathrm{d}}{\mathrm{d}t}X(t, \beta) = F(t, \beta, X(t,\beta), u(t)),
  \]
  where $F$ is a smooth vector field on $\mathcal{F}(K)$, and $X(t, \cdot)\in \mathcal{F}(K)$ is the state, which is an integral curve of $F$. This ensemble is said to be ensemble controllable on $\mathcal{F}(K)$, if for any $\epsilon>0$ and arbitrary desired target state $X_F(\cdot )\in \mathcal{F}(K)$, starting with any initial state $X_0(\cdot )\in \mathcal{F}(K)$, there exists a piecewise constant control signal $u:[0, T]\to \mathbb{R}^m$ that steers the ensemble into an $\epsilon$-neighborhood of $X_F(\cdot )$ at a finite time $T > 0$, i.e., $d(X(T, \cdot), X_F(\cdot )) < \epsilon$, where $d: \mathcal{F}(K)\times \mathcal{F}(K) \to \mathbb{R}$ is a metric on $\mathcal{F}(K)$. 
  
  In particular, if we associate the state space $\mathcal{F}(K)$ with the uniform metric, i.e., $d(f,g)=\mathrm{sup}_{\beta \in K} \rho (f(\beta), g(\beta))$ for any $f, g\in \mathcal{F}(K)$, where $\rho: M\times M\to \mathbb{R}$ is a metric on $M$, then \cref{def: ensemble controllability} is referred as the \emph{uniform ensemble controllability}.
\end{definition}

Throughout the paper, we focus on the case where $K$ is a compact subset of $\mathbb{R}$, $\mathcal{F}(K) = C(K, \mathbb{R}^{n})$ is the space of $n$-dimensional real-valued continuous functions, and $d$ is the metric induced by the sup-norm, i.e., $d(f, g) = \|f-g\|_{\infty } = \mathrm{sup}_{\beta\in K} \|f(\beta)-g(\beta)\|$, where $\|\cdot \|$ is any norm on $\mathbb{R}^n$. We specifically analyze UEC of linear ensembles indexed by $\beta \in K$, characterized by,
\begin{equation}
  \label{equ: sys}
  \frac{\mathrm{d}}{\mathrm{d}t} X(t, \beta ) = A(\beta )X(t, \beta ) + B(\beta ) u(t),
\end{equation}
where $X(t, \cdot ) \in C(K, \mathbb{R}^n)$ is the state variable; $A( \cdot ) \in C(K, \mathbb{R}^{n\times n})$ is the drift vector field with $A(\beta)$ being diagonalizable for all $\beta\in K$; $B( \cdot )\in C(K, \mathbb{R}^{n\times m})$ is the control vector field; and $u :[0, T]\to \mathbb{R}^m$ is the control signal which is piecewise constant. We assume that the eigenvalues of $A(\cdot )$ are all \emph{finite-to-one}, i.e., the inverse image of $\eta\in \lambda(K)$ with $\lambda(\beta)$ as an eigenvalue of $A(\beta)$, $\beta\in K$, always has finite cardinality. 

\Cref{fig:relations among ensembles} illustrates the relations among the four kinds of linear ensemble systems, namely, R-, C-, RR-, and RC-ensembles, that are commonly discussed in ensemble control literature. In particular, if the control signals $u(t)$ are real-valued, and furthermore, the state variable $X(t, \cdot )$, drift and control vector fields $A(\cdot )$ and $B(\cdot )$, are all real-valued, then the ensemble in \cref{equ: sys} is called an R-ensemble; otherwise, such an ensemble is called a C-ensemble. Depending on how the eigenvalues of the drift vector field take value, each R-ensemble is either an RR-ensemble or an RC-ensemble. Specifically, if the eigenvalues of $A(\beta)$ are all real for any $\beta\in K$, then such R-ensemble is further called an RR-ensemble; otherwise it is called an RC-ensemble. Each C-ensemble can be transformed into an equivalent R-ensemble (could be either RR- or RC-ensemble) by considering the real and imaginary parts of its state variable separately; while not every R-ensemble can be turned into an equivalent C-ensemble (see \cref{ex: motivating example 2}).

\begin{figure}[htbp]
  \centering
  \includegraphics[width = 0.6\linewidth]{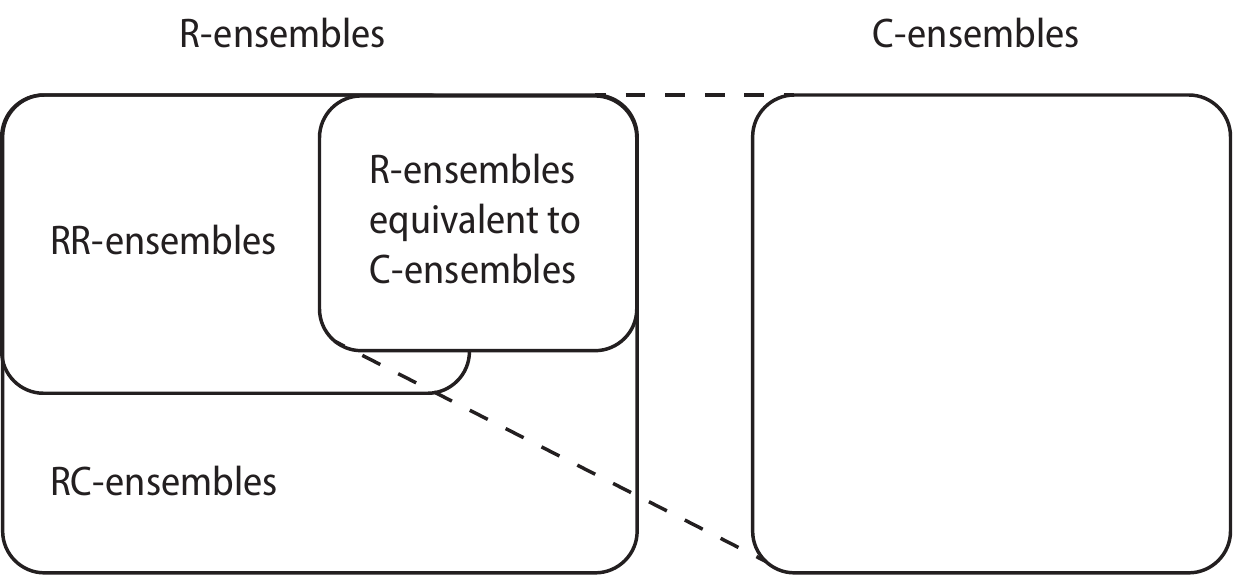}
  \caption{Relations among R-, RR-, RC-, and C-ensembles} \label{fig:relations among ensembles}
\end{figure}

In \cite{li2019separating}, we have fully characterized UEC of RR-ensembles through the lens of separating points of the control vector field, i.e., $B(\cdot )$ in \cref{equ: sys}. To illustrate the idea, let us first consider a one-dimensional RR-ensemble indexed by the parameter $\beta$ varying on a compact set $K\subset \mathbb{R}$, characterized by
\begin{equation}
  \label{equ: one-dim sysm}
  \frac{\mathrm{d}}{\mathrm{d}t} x(t, \beta) = a(\beta)x(t, \beta) + \sum_{i=1}^m b_i(\beta)u_i(t),
\end{equation}
where $x(t, \cdot)\in C(K, \mathbb{R})$, $a( \cdot ), b_i( \cdot )\in C(K, \mathbb{R})$, and $u_i:[0, T]\to \mathbb{R}$, $i\in [1:m]$, are piecewise constant control inputs. It is proved (see \cite[Corollary~1]{li2019separating}) that this ensemble is uniformly ensemble controllable if and only if its $s$-separation matrix, given by
\begin{equation}
  \label{equ: ensemble.controllability.matrix.Li.20}
  D(s) = \begin{bmatrix}
  b_1(a^{-1}_1(s)) & \cdots & b_m(a^{-1}_1(s)) \\
  \vdots & \ddots & \vdots\\
  b_1(a^{-1}_{\kappa (s)}(s)) & \cdots & b_m(a^{-1}_{\kappa (s)}(s))
  \end{bmatrix},
\end{equation}
has full row-rank, i.e., $\rank D(s) = \kappa(s)$, for all $s\in a(K)$, where $a(K)$ is the image of $K$ under the map $a: K\to \mathbb{R}$, and the set $\{a^{-1}_{1}(s), \ldots , a^{-1}_{\kappa (s)}(s)\}$ is the preimage of $s$ under the map $a$, with $\kappa(s)$ denotes its cardinality.

\begin{remark}
  Although the matrix $D(s)$ in \cref{equ: ensemble.controllability.matrix.Li.20} was called the ``ensemble controllability matrix'' for RR-ensembles in \cite{li2019separating}, we will introduce a new concept of ensemble controllability matrix for general R-ensembles in Section~\ref{sec: UEC of Linear Ensembles through Stone-Weierstrass Theorem for Modules}, which also involves information of the drift vector field $A(\cdot )$. Therefore, to avoid abuse of notations, we will refer the matrix $D(s)$ in \cref{equ: ensemble.controllability.matrix.Li.20} as the \textit{$s$-separation matrix} associated with the ensemble in \cref{equ: one-dim sysm}.
\end{remark}

Now, let us consider the case when the ensemble in \cref{equ: sys} is an $n$-dimensional RR-ensemble, i.e., the eigenvalues of $A(\beta)$, denoted as $\lambda_1(\beta), \ldots , \lambda_n(\beta)$, are real for all $\beta\in K$. Without loss of generality, we assume $A(\beta)$ is an upper-triangular matrix for all $\beta\in K$; otherwise there exists a family of invertible matrices, denoted as $T(\beta) \in C(K, \mathbb{R}^{n\times n})$, such that $T^{-1}(\beta) A(\beta) T(\beta)$ is upper-triangular for all $\beta\in K$. Then, it is proved (see Theorem 4 in \cite{li2019separating}) that the ensemble in \cref{equ: sys}, with its state variable $X(t, \cdot ) \in C(K, \mathbb{R}^n)$ defined in an infinite-dimensional space, is uniformly ensemble controllable if and only if the following system in \cref{equ: real eigenvalue induced system}, with its state variable $(Z_1, \ldots , Z_n)$ defined in a finite-dimensional Euclidean space, given by 
\begin{equation}
  \label{equ: real eigenvalue induced system}
  \frac{\mathrm{d}}{\mathrm{d}t}
  \begin{bmatrix}
    Z_1(t, s_1) \\ \vdots \\ Z_n(t, s_n)
  \end{bmatrix} =
  \begin{bmatrix}
    s_1 I_{\kappa_1(s_1)} \\ & \ddots \\ & & s_n I_{\kappa_n(s_n)}
  \end{bmatrix}
  \begin{bmatrix}
    Z_1(t, s_1) \\ \vdots \\ Z_n(t, s_n)
  \end{bmatrix} + 
  \begin{bmatrix}
    D_1(s_1) \\ \vdots \\ D_n(s_n)
  \end{bmatrix}U(t),
\end{equation}
is controllable on $\mathbb{R}^N$ for each $n$-tuple $(s_1, \ldots , s_n)\in K_1\times \cdots \times K_n$, where $K_i = \lambda_i(K)$ is the image of $K$ under $\lambda_i$; $\kappa_i(s_i)$ is the cardinality of the preimage of $s_i$ under $\lambda_i$; $N = \sum_{i=1}^n \kappa_i(s_i)$; $I_{\kappa_i(s_i)}$ is the $\kappa_i(s_i)\times \kappa_i(s_i)$ identity matrix; and $D_i(s_i) \in \mathbb{R}^{\kappa_i(s_i)\times m}$ is the $s_i$-separation matrix associated with
\[
  \frac{\mathrm{d}}{\mathrm{d}t}z_i(t, \beta) = \lambda_i(\beta)z_i(t,\beta) + b_i(\beta)U(t),
\]
with $b_i(\beta)$ denotes the $i$\textsuperscript{th} row of $B(\beta)$.

To summarize, \cite{li2019separating} provides a sufficient and necessary condition based on separating properties of the control vector field to check UEC of RR-ensemble. In this work, we will extend and upgrade the scope of \cite{li2019separating} by proposing a new algebraic framework based on Stone-Weierstrass theorem for modules to check UEC of general R-ensembles. We will provide a novel concept of ensemble controllability matrix that renders a sufficient and necessary rank-condition for UEC of R-ensembles.

\section{UEC of R-Ensembles through Stone-Weierstrass Theorem for Modules}
\label{sec: UEC of Linear Ensembles through Stone-Weierstrass Theorem for Modules}
In this section, leveraging the powerful but recondite Stone-Weierstrass theorem for modules, we propose a new framework to examine UEC of R-ensembles. We provide a sufficient and necessary condition on checking UEC through checking controllability of subsystems in the ensemble, which can be achieved by testing the rank-condition of ensemble controllability matrix. First, we introduce a set of useful definitions. In what follows, let $\mathbb{F}$ denote a field which is either $\mathbb{R}$ or $\mathbb{C}$.

\begin{definition}[Equivalence and Equivalent Class]
	Let $V$ be a locally convex vector space over $\mathbb{F}$ and $K \subset \mathbb{F}$ be compact. Given a non-empty $S\subset C(K, V)$ and $x, y\in K$, $x$ is said to be equivalent to $y$ modulo $S$, denoted as $x \equiv y \pmod S$, if $f(x) = f(y)$ for all $f\in S$. Furthermore, for every $x\in K$, the subset $[x]_{S}=\{y\in K: y\equiv x \pmod S\}\subset K$ containing $x$ is called an equivalent class modulo $S$.
\end{definition}

\begin{definition}[Module]
	Let $X$ be a Hausdorff space, $V$ be a locally convex vector space over $\mathbb{F}$, and $\mathcal{A}$ be a subalgebra of $C(X, \mathbb{F})$. A vector subspace $W \subset C(X, V)$ is called an $\mathcal{A}$-module, if for every $f\in \mathcal{A}$ and $g\in W$, the function $fg:X \to V$, $x\mapsto f(x)g(x)$ satisfies $fg\in W$.
\end{definition}

\begin{definition}[Localizability]
  Let $V$ be a locally convex vector space over a field $\mathbb{F}$, $K \subset \mathbb{F}$ be compact, $\mathcal{A}$ be a subalgebra of $C(K, \mathbb{F})$, and $W\subset C(K, V)$ be an $\mathcal{A}$-module. Then, we say a function $f\in C(K, V)$ is localizable by $W$ under $\mathcal{A}$ if for any $\epsilon>0$ and every equivalent class $E \subset K$ modulo $\mathcal{A}$, there exists $g\in W$ such that $\sup_{x\in E}\Vert f(x)-g(x)\Vert <\epsilon$.
\end{definition}


\begin{theorem}[Stone-Weierstrass Theorem for Real-Valued Modules] \label{thm: Stone-Weierstrass for modules}
	Let $X$ be a Hausdorff space, $V$ be a locally convex vector space over $\mathbb{R}$, $\mathcal{A}$ be a subalgebra of $C(K, \mathbb{R})$, and $W \subset C(X, V)$ be a vector subspace which is an $\mathcal{A}$-module. Then, for any $f\in C(X, V)$, it holds that $f\in \bar{W}$ if and only if $f$ is localizable by $W$ under $\mathcal{A}$, where $\bar{W}$ is the closure of $W$.
\end{theorem}

\begin{proof}
	See \cite[pp.~15-17]{Prolla1993WeierstrassStoneTT}.
\end{proof}

\begin{remark}
  \Cref{thm: Stone-Weierstrass for modules} is a generalization of the commonly known Stone\hyp{}Weierstrass theorem for real-valued functions. Specifically, if $\mathcal{A}$ is a subalgebra that separates points (i.e., for any $x, y\in K$ satisfying $x\neq y$, there exists $f\in \mathcal{A}$ such that $f(x)\neq f(y)$), then each equivalent class of $K$ modulo $\mathcal{A}$ is a singleton so that a function $f\in C(X, V)$ is localizable by $W$ under $\mathcal{A}$ if $W$ contains the constant function $\mathds{1}(\cdot)$. If it further holds that $V = \mathbb{R}$ and $W = \mathcal{A}$, then \cref{thm: Stone-Weierstrass for modules} boils down to the Stone-Weierstrass theorem for real-valued functions which asserts that a subalgebra of continuous real-valued functions that separates points and contains constant functions is dense in $C(X, \mathbb{R})$.
\end{remark}

Next, we extend \cref{thm: Stone-Weierstrass for modules} to the complex-valued cases.

\begin{lemma}
  \label{lem: complex module to real module}
	Let $X$ be a Hausdorff space, and $\mathcal{A}$ be a self-adjoint subalgebra of $C(X, \mathbb{C})$, i.e., $\mathcal{A}$ is closed under complex conjugate. Suppose $V$ is a locally convex space over $\mathbb{C}$, then $W \subset C(X, V)$ is an $\mathcal{A}$-module if and only if $W$ is a $\mathcal{B}$-module, where $\mathcal{B} := \{\Re(a): a\in \mathcal{A}\}$.
\end{lemma}

\begin{proof}
	(Sufficiency): It suffices to show that $\mathcal{B} \subset \mathcal{A}$ and $\mathcal{B}$ is a subalgebra of $C(X, \mathbb{R})$. It is evident that $\mathcal{B} \subset \mathcal{A}$ since from $\mathcal{A}$ being self-adjoint, we have $\Re(a) = \frac{1}{2}(a+ \bar{a}) \in \mathcal{A}$ for any $a\in \mathcal{A}$. To prove $\mathcal{B}$ is a subalgebra of $C(X, \mathbb{R})$, we observe that for any $a_1,a_2 \in \mathcal{A}$, we have $\Re(a_1), \Re(a_2) \in \mathcal{A}$ so that $\Re(a_1)\Re(a_2)\in \mathcal{A}$. Then, it holds that $\Re(a_1)\Re(a_2) = \Re{}[\Re(a_1)\Re(a_2)] \in \mathcal{B}$, which implies $\mathcal{B}$ is a subalgebra of $C(X, \mathbb{R})$.

	(Necessity): Let $a\in \mathcal{A}$ and write $a = u+iv$. By the definition of $\mathcal{B}$, we have $u = \Re(a) \in \mathcal{B}$. Furthermore, since $v = \Re(-ia)$, we have $v\in \mathcal{B}$. Now, if $W$ is a $\mathcal{B}$-module, for any $f\in W$ we have $af = uf + ivf\in W$, which implies that $W$ is an $\mathcal{A}$-module.
\end{proof}

\begin{cor}
  \label{cor: Stone-Weierstrass for complex modules}
	Let $X$ be a Hausdorff space, $V$ be a locally convex vector space over the field $\mathbb{C}$, $\mathcal{A}$ be a self-adjoint subalgebra of $C(K, \mathbb{C})$, and $W \subset C(X, V)$ be a vector subspace which is an $\mathcal{A}$-module. Let $\mathcal{B}:=\Re(\mathcal{A})=\{\Re(f):f\in\mathcal{A}\}$. Then, for each $f\in C(X, V)$, it holds that $f\in \bar{W}$ if and only if $f$ is localizable by $W$ under $\mathcal{B}$, where $\bar{W}$ is the closure of $W$.
\end{cor}

\begin{proof}
	The proof is evident by \cref{thm: Stone-Weierstrass for modules} and \cref{lem: complex module to real module}.
\end{proof}

In the remaining of this section, we demonstrate that Stone-Weierstrass theorem for modules enables an elegant and systematic approach to examine UEC of R-ensembles through checking controllability of their subsystems. The key idea is to rewrite the reachable set of an R-ensemble into a module. Then, if the reachable set is rich enough so that any continuous function over $K$ is localizable by the corresponding module, we shall conclude that the R-ensemble is uniformly ensemble controllable.

To fix ideas, let us consider the R-ensemble indexed by $\beta\in K$, given by
\begin{equation}
  \label{equ: diagonal.linear.ensemble.general.eigenvalue}
  \begin{aligned}
    \frac{\mathrm{d}}{\mathrm{d}t} X(t,\beta) &= A(\beta)X(t,\beta)+B(\beta)u(t) \\
    &= \diag\{\lambda_1(\beta), \dots, \lambda_{r}(\beta), G_1(\beta), \dots, G_{l}(\beta)\}X(t, \beta)+\sum_{j=1}^m b_j(\beta)u_j(t),
  \end{aligned}
\end{equation}
where $G_j = \begin{bsmallmatrix} \alpha_j(\beta) & -\omega_j(\beta) \\ \omega_j(\beta) & \alpha_j(\beta) \end{bsmallmatrix}$, $j\in [1:l]$, and $u_j:[0, T]\to \mathbb{R}$ are piecewise constant control signals for $j \in [1:m]$. In this case, the drift vector field $A(\beta)$ in \cref{equ: diagonal.linear.ensemble.general.eigenvalue} is said to be in its ``real Jordan form''. It is proved in \cite[pp.~202]{horn2012matrix} that for every real-valued $A$ that is diagonalizable in $\mathbb{C}$, there exists an invertible matrix $T\in \mathbb{R}^{n\times n}$ such that $T^{-1}AT$ is in the real Jordan form. Therefore, the UEC analysis of the ensemble in \cref{equ: diagonal.linear.ensemble.general.eigenvalue} in this section can be generalized to R-ensembles with diagonalizable drift vector fields.

For ease of exposition, we denote
\[
  \mathcal{K} = \bigcup_{j=1}^r \lambda_j(K) \bigcup_{j=1}^l (\alpha_j+i\omega_j)(K) \bigcup_{j=1}^l (\alpha_j - i\omega_j)(K).
\]
Given any $\eta\in \mathcal{K}$, we denote
\[
  \gamma^{-1}(\eta) = \bigcup_{j=1}^r \{\lambda_j^{-1}(\eta)\} \bigcup_{j=1}^l \{(\alpha_j+i\omega_j)^{-1}(\eta)\}\bigcup_{j=1}^l \{(\alpha_j-i\omega_j)^{-1}(\eta)\},
\]
where $\lambda_j^{-1}(\eta)$, $(\alpha_j+i\omega_j)^{-1}(\eta)$, and $(\alpha_j-i\omega_j)^{-1}(\eta)$ are the preimages of $\eta$ under the map $\lambda_j$, $\alpha_j + i\omega_j$, and $\alpha_j - i\omega_j$, respectively. We denote the cardinality of $\gamma^{-1}(\eta)$ as $\kappa(\eta)$. Given a matrix $M\in \mathbb{R}^{n\times m}$ and $S = \{\beta_1, \ldots , \beta_p\}\subset K$, we denote $\mathcal{D}_M(S) \in \mathbb{R}^{pn \times pm}$ and $\mathcal{V}_M(S) \in \mathbb{R}^{pn \times m}$ as
\[
  \mathcal{D}_M(S) =
  \begin{bmatrix}
    M(\beta_1) \\ & \ddots \\ & & M(\beta_p)
  \end{bmatrix} \quad \text{and} \quad \mathcal{V}_M(S) =
  \begin{bmatrix}
    M(\beta_1) \\ \vdots \\ M(\beta_p)
  \end{bmatrix},
\]
respectively.

\begin{theorem}
  \label{thm: UEC.general.complex.module.self.adjoint}
	Suppose the reachable set of the ensemble in \cref{equ: diagonal.linear.ensemble.general.eigenvalue}, characterized by
  \[
    \mathcal{L} = \mathrm{span}_{\mathbb{R}}\{A^kb_j: j \in [1:m], k\in \mathbb{N}\},
  \]
  is closed under left multiplication of $A^\ast$, where
  \begin{equation}
    \label{equ: UEC.general.complex.module.self.adjoint.varphi}
		A^{\ast} = \begin{bmatrix} \lambda_1(\cdot ) \\ & \ddots \\ & & \lambda_r(\cdot ) \\ & & & G^{\ast}_1(\cdot ) \\ & & & & \ddots \\ & & & & &  G^{\ast}_l(\cdot )
		\end{bmatrix},
  \end{equation}
  i.e., for any $f(\cdot)\in\mathcal{L}$, $A^{\ast}(\cdot)f(\cdot)\in\mathcal{L}$, then the ensemble in \cref{equ: diagonal.linear.ensemble.general.eigenvalue} is uniformly ensemble controllable on $C(K, \mathbb{R}^{n})$ if and only if for any $\eta\in \mathcal{K}$, the system given by
	\begin{equation}
    \label{equ: subsys.general.complex.module.self.adjoint}
		\frac{\mathrm{d}}{\mathrm{d}t} Y(t, \eta) = \mathcal{D}_A (\gamma^{-1}(\eta)) Y(t, \eta) + \mathcal{V}_B (\gamma^{-1}(\eta))u(t) 
	\end{equation}
	is controllable on $\mathbb{R}^{n\kappa (\eta)}$.
  where $n = r+2l$ and $G_j^{\ast}(\beta) = \begin{bsmallmatrix} \alpha_j(\beta) & \omega_j(\beta) \\ -\omega_j(\beta) & \alpha_j(\beta) \end{bsmallmatrix}$, $j \in [1:l]$ for all $\beta\in K$. 
\end{theorem}

\begin{proof}
	(Sufficiency):
  Let $P = \begin{bsmallmatrix} I_r \\ & I_{l}\otimes T	\end{bsmallmatrix}$, where $T = \tfrac{1}{\sqrt{2}} \begin{bsmallmatrix} 1 & 1 \\ i & -i \end{bsmallmatrix}$ and $\otimes$ denotes the Kronecker product of two matrices. We can verify that
  \[
    \Lambda:= P^{-1}AP =
    \begin{bNiceMatrix}
      \lambda_1 \\ & \ddots \\ & & \lambda_r \\ & & &  \alpha_1 - i\omega_1 \\ & & & & \alpha_1 + i\omega_1 \\ & & & & & \ddots \\ & & & & & &  \alpha_l - i\omega_l \\ & & & & & & &  \alpha_l + i\omega_l
	  \end{bNiceMatrix}.
  \]
	Hence, the reachable set of the ensemble in \cref{equ: diagonal.linear.ensemble.general.eigenvalue}, i.e., $\mathcal{L}$, is isomorphic to
  \[
    \tilde{\mathcal{L}} = \mathrm{span}_{\mathbb{R}}\{P^{-1}A^k b_j: j \in [1:m], k\in \mathbb{N}\} = \mathrm{span}_{\mathbb{R}} \{\Lambda^k \tilde{b}_j: j\in [1:m], k \in \mathbb{N}\},
  \]
  with $\tilde{b}_j = P^{-1}b_j$.

	We first prove that if $\mathcal{L}$ is 
  closed under left multiplication of $A^{\ast}$, then $\tilde{\mathcal{L}}$ is closed under left multiplication of $\bar{\Lambda}$, where $\bar{\Lambda}$ denotes the complex-conjugate of $\Lambda$. Note that $\bar{A}=A$, and since it is straightforward to verify that $P^{-1}A^{\ast}P = \bar{P}^{-1}A\bar{P}$, we have
  \[
    \begin{aligned}
      \bar{\Lambda}\Lambda^k \tilde{b}_j &= \overline{P^{-1} A P} (P^{-1}A P)^k P^{-1}b_j = \bar{P}^{-1} A \bar{P} P^{-1} A^k b_j \\
      &= P^{-1}A^{\ast}P P^{-1} A^k b_j = P^{-1}A^{\ast}A^kb_j,
    \end{aligned}
  \]
  for any $\Lambda^{k}\tilde{b}_j\in\tilde{\mathcal{L}}$. If $\mathcal{L}$ is closed under left multiplication of $A^{\ast}$, then $A^{\ast}A^k b_j$ is in the form of a finite linear combination, i.e., $A^{\ast}A^kb_j=\sum_{s}c_sA^{k_s}b_{j_s}$, with $k_{s}\in\mathbb{N}$ and $j_s\in[1:m]$.
  Therefore, $\bar{\Lambda}\Lambda^k \tilde{b}_j = P^{-1}A^{\ast}A^kb_j = \sum_{s}c_sP^{-1}A^{k_s}b_{j_s} \in \tilde{\mathcal{L}}$. 
  
  Next, let us consider the map $\psi: \tilde{\mathcal{L}} \to C(K\times I, \mathbb{C})$ with $I:=[1:r+2l]$, which is defined as follows: for any $f\in\tilde{\mathcal{L}}$,
  \[
    (\psi f)(\beta,s) = \langle f \rangle_{s}(\beta),
  \]
	where $\beta\in K$ and $\langle f \rangle_{s}$ denotes the $s$\textsuperscript{th} component of $f$ for $s\in I$. Obviously, $\psi$ is a one-to-one bounded linear map, so its image $W:= \psi(\tilde{\mathcal{L}})$ is a vector subspace of $C(K\times I, \mathbb{C})$ that is isomorphic to $\tilde{\mathcal{L}}$. For any $f = \Lambda^k \tilde{b}_j\in \tilde{\mathcal{L}}$, we have
  \[ 
    (\psi f)(\beta,s)=
    \begin{dcases}
      \lambda_s^k \langle\tilde{b}_j\rangle_{s}(\beta), & \text{if } s \in [1:r], \\
		  (\alpha_{q}-i\omega_{q})^k \langle\tilde{b}_j\rangle_{s}(\beta), & \text{if } s=r-1+2q \in [r+1:2:r+2l-1], \\
		  (\alpha_{q}+i\omega_{q})^k \langle\tilde{b}_j\rangle_{s}(\beta), & \text{if } s=r+2q \in [r+2:2:r+2l].
    \end{dcases} 
  \]
	Let $\mathcal{B}=\mathrm{alg}_{\mathbb{R}} \{\xi , \bar{\xi}\} \subset C(K\times I, \mathbb{C})$ denote the \emph{real} algebra generated by $\xi$ and $\bar{\xi}$, where $\xi$ is defined as
  \[
    \xi(\beta,s)=
    \begin{dcases}
      \lambda_s(\beta), & \text{if } s \in [1:r], \\
		  (\alpha_{q} - i\omega_{q})(\beta), & \text{if } s=r-1+2q \in [r+1:2:r+2l-1], \\
		  (\alpha_{q} + i\omega_{q})(\beta), & \text{if } s=r+2q \in [r+2:2:r+2l].
    \end{dcases}
  \]
	It is straightforward that $\mathcal{B}$ is a self-adjoint subalgebra of $C(K\times I, \mathbb{C})$. Moreover, we observe that the image $W=\psi(\tilde{\mathcal{L}})$ is a module of $\mathcal{B}$: obviously, $W$ is closed under multiplication of $\xi$; and for any $f\in\tilde{\mathcal{L}}$, $\bar{\xi}\cdot\psi(f)=\psi(\bar{\Lambda}f)\in\psi(\tilde{\mathcal{L}})$ because $\tilde{\mathcal{L}}$ is closed under left multiplication of $\bar{\Lambda}$. As a result, by \cref{lem: complex module to real module}, $W$ is also a module of the real subalgebra $\mathcal{C} \subset C(K\times I , \mathbb{R})$, where $\mathcal{C}:= \{\Re(f): f\in \mathcal{B}\}$. 

	Next, we examine the equivalent classes of $K\times I$ characterized by $\mathcal{C}$. First, note that $\{\xi+\bar{\xi}, \xi-\bar{\xi}\}$ is another set of generators of $\mathcal{B}$, so the subalgebra $\mathcal{C}=\Re(\mathcal{B})$ is spanned by $\{\Re[(\xi+\bar{\xi})^{k_1}(\xi-\bar{\xi})^{k_2}]: k_1, k_2\in\mathbb{N} \text{ and } k_1+k_2>0\}$. Since the value of $(\xi\pm\bar{\xi})$ is either real or pure imaginary, we can verify easily that $\mathcal{C}=\mathrm{alg}_{\mathbb{R}}\{\mu_1, \mu_2\}$, where
  \[
    \mu_1(\beta,s)=
    \begin{dcases}
      \lambda_{s}(\beta), & \text{if}\ s\in[1:r], \\
		  \alpha_{q}(\beta), & \text{if}\ s\in[r+1:r+2l],
    \end{dcases}
  \]
  and
  \[
    \mu_2(\beta,s)=
    \begin{dcases}
      0, & \text{if}\ s\in[1:r], \\
      \omega_{q}^{2}(\beta), & \text{if}\ s\in[r+1:r+2l],
    \end{dcases}
  \]
  where $q=\lceil\tfrac{s-r}{2}\rceil$. Consequently, if $(\beta_1, s_1) \equiv (\beta_2, s_2) \pmod{\mathcal{C}}$, then one of the following three cases happens.

	\begin{enumerate}[label=(\arabic*), leftmargin= 2em]
		\item Both $s_1, s_2\in [1:r]$. In this case, $\mu_2=0$, and $\mu_1(\beta_1, s_1)=\mu_2(\beta_2, s_2)$ gives that \[\lambda_{s_1}(\beta_1) = \lambda_{s_2}(\beta_2).\]

    \item Both $s_1, s_2\in [r+1:r+2l]$. In this case, $\mu_1(\beta_1, s_1)=\mu_1(\beta_2, s_2)$ implies $\alpha_{q_1}(\beta_1)=\alpha_{q_2}(\beta_2)$, and $\mu_2(\beta_1, s_1)=\mu_2(\beta_2, s_2)$ implies $\omega_{q_1}(\beta_1)=\pm\omega_{q_2}(\beta_2)$. So we have that \[(\alpha_{q_1}+i\omega_{q_1})(\beta_1) = (\alpha_{q_2} \pm i\omega_{q_2})(\beta_2).\]

		\item $s_1\in [1:r]$ and $s_2\in [r+1:r+2l]$. In this case, $\mu_1(\beta_1, s_1)=\mu_1(\beta_2, s_2)$ gives that $\lambda_{s_1}(\beta_1)=\alpha_{q_2}(\beta_2)$, and $\mu_2(\beta_1, s_1)=\mu_2(\beta_2, s_2)$ gives that $\omega_{q_2}(\beta_2)=0$. In other words, we obtain that \[\lambda_{s_1}(\beta_1) = (\alpha_{q_2}\pm i\omega_{q_2})(\beta_2).\]
	\end{enumerate}
	Based on the above analysis, we conclude that the equivalent classes of $K\times I$ modulo $\mathcal{C}$ is determined by the spectra of $A$. Specifically, let us  fix an $\eta\in \mathcal{K}$, and without loss of generality, we assume that $\eta = \lambda_j(\beta)$ for some $j \in [1:r]$ and $\beta\in K$. Then, the equivalent class $[(\beta, j)]_{\mathcal{C}}$ is given by $[(\beta, j)]_{\mathcal{C}} = S_1 \cup S_2 \cup S_3$, where $S_1 = \bigcup_{j=1}^{r} \{(\lambda_j^{-1}(\eta), j)\}$, $S_2 = \bigcup_{j=1}^{l} \{((\alpha_j +i\omega_j)^{-1}(\eta), j)\}$, and $S_3 = \bigcup_{j=1}^{l} \{((\alpha_j - i\omega_j)^{-1}(\eta), j)\}$. 
	
	Now let us consider the linear system in \cref{equ: subsys.general.complex.module.self.adjoint}, and suppose that
  \[
    \frac{\mathrm{d}}{\mathrm{d}t} Y(t,\eta) = \mathcal{D}_{A}(\gamma^{-1}(\eta)) Y(t,\eta) + \mathcal{V}_{B}(\gamma^{-1}(\eta)) u(t) 
  \]
  is controllable on $\mathbb{R}^{n\kappa (\eta)}$, then we have
  \[
    \begin{aligned}
      &\mathrm{span}_{\mathbb{R}}\{ \mathcal{D}_{A}^k (\gamma^{-1}(\eta))\mathcal{V}_{b_j}(\gamma^{-1}(\eta)): j\in [1:m], k\in \mathbb{N}\} \\
      &=\mathrm{span}_{\mathbb{R}}\{\mathcal{V}_{A^k b_j}(\gamma^{-1}(\eta)): j\in [1:m], k\in \mathbb{N}\} = \mathbb{R}^{n\kappa (\eta)}.
    \end{aligned}
  \]
	Hence, by identifying $\mathbb{R}^{n\kappa(\eta)}$ with $C(\gamma^{-1}(\eta),\mathbb{R}^{n})$ through the isomorphism
  \[
    \Theta: \mathbb{R}^{n\kappa (\eta)}\to C(\gamma^{-1}(\eta), \mathbb{R}^n), \quad \Theta(e_{n(j-1)+k})=\mathds{1}_{\beta_j}\hat{e}_k,
  \]
  for $j\in [1: \kappa (\eta)]$ and $k \in [1:n]$, where $\{e_k\}$ and $\{\hat{e}_k\}$ denote the standard basis of $\mathbb{R}^{n\kappa(\eta)}$ and $\mathbb{R}^n$, respectively, and $\mathds{1}_{\beta_j}$ is the characteristic function at $\beta_j\in \gamma^{-1}(\eta)$, it is obvious that
	\[
		\mathcal{L}\vert_{\gamma^{-1}(\eta)} = \mathrm{span}_{\mathbb{R}} \{A^k b_j(\gamma^{-1}(\eta)): j\in [1:m], k\in\mathbb{N}\}
	\]
	is dense in $C(\gamma^{-1}(\eta), \mathbb{R}^n)$.

	Now, given any $f\in W$ and $\epsilon>0$, we have $(P\psi^{-1})f\in \mathcal{L}$. Since $\mathcal{L}\vert_{\gamma^{-1}(\eta)}$ is dense in $C(\gamma^{-1}(\eta), \mathbb{R}^n)$, we can find $g\in C(K, \mathbb{R}^n)$ such that
  \[
    \|(P\psi^{-1})f \vert_{\gamma^{-1}(\eta)} - g\vert_{\gamma^{-1}(\eta)}\|<\epsilon.
  \]
  If we use the map $\pi : K\times I \to K$, $\pi(\beta, j)=\beta$ to denote the projection of $(\beta, j)$ onto the first index, then it is clear that $\pi ([(\beta, j)]_{\mathcal{C}}) = \gamma^{-1}(\eta)$. Furthermore, by the definition of $\gamma^{-1}(\eta)$, there exists a permutation of $[1:\kappa(\eta)]$, denoted as $j_1, \ldots , j_{\kappa(\eta)}$, such that for each $\beta_s\in \gamma^{-1}(\eta)$, if we associate it with $j_s, s = 1, \ldots , \kappa(\eta)$ then we have $[\beta, j]_{\mathcal{C}} = \{(\beta_s, j_s): s\in [1:\kappa(\eta)]\} $. In this way, a 1-to-1 correspondence between $\gamma^{-1}(\eta)$ and $[\beta, j]_{\mathcal{C}}$ is constructed and therefore we have $\|f \vert_{[\beta, j]_{\mathcal{C}}} - (\psi P^{-1})g \vert_{[\beta, j]_{\mathcal{C}}} \| < \|\psi\|\|P^{-1}\|\epsilon$, where $\|P^{-1}\|$ is bounded since $\lambda_j$'s and $(\alpha_j \pm i\omega_j)$'s are all continuous functions over the compact set $K$. Since $\eta$ is arbitrary, the uniform approximation of $f$ can be achieved on all equivalent classes of $K$ characterized by $\mathcal{C}$ so that $f$ is localizable by $\psi P^{-1}(C(K, \mathbb{R}^n))$. Therefore, by Stone-Weierstrass theorem, $W$ is dense in $\psi P^{-1}(C(K, \mathbb{R}^n))$, which concludes that $\mathcal{L} = P(\psi^{-1}(W))$ is dense in $C(K, \mathbb{R}^{n})$.
	
	(Necessity): We observe that the system in \cref{equ: subsys.general.complex.module.self.adjoint} is the collection of subsystems of the ensemble indexed by $\gamma^{-1}(\eta)$. Therefore, it is straightforward that if the ensemble in \cref{equ: diagonal.linear.ensemble.general.eigenvalue} is ensemble controllable on $C(K, \mathbb{R}^{n})$, then the system in \cref{equ: subsys.general.complex.module.self.adjoint} is controllable on $\mathbb{R}^{n \kappa (\eta)}$.
\end{proof}

When all eigenvalues of $A(\cdot )$, i.e., $\lambda_i(\cdot )$, $i = 1, \dots, r$, and $(\alpha_j \pm i\omega_j)(\cdot )$, $j = 1, \dots , l$, are finite-to-one, it holds that $\kappa(\eta) < \infty $ for all $\eta\in \mathcal{K}$ so that the system in \cref{equ: subsys.general.complex.module.self.adjoint} is always a finite-dimensional system, which controllability can be examined through rank-condition. We write $\gamma^{-1}(\eta) = \{\beta_1, \ldots , \beta_{\kappa(\eta)}\}$. Then, we call the matrix $G(\eta)\in \mathbb{R}^{n\kappa(\eta)\times mn\kappa(\eta)}$, defined by
\[
  \begin{aligned}
    G(\eta) &= \bigl[\mathcal{V}_B(\gamma^{-1}(\eta), \mathcal{D}_A (\gamma^{-1}(\eta)) \mathcal{V}_B (\gamma^{-1}(\eta)), \ldots, \mathcal{D}^{n\kappa(\eta)-1}_A (\gamma^{-1}(\eta)) \mathcal{V}_B (\gamma^{-1}(\eta))\bigr] \\
    &=
    \begin{bmatrix}
      B(\beta_1) & A(\beta_1)B(\beta_1), &\ldots & A^{n\kappa(\eta)-1}(\beta_1)B(\beta_1) \\
      B(\beta_2) & A(\beta_2)B(\beta_2), &\ldots & A^{n\kappa(\eta)-1}(\beta_2)B(\beta_2) \\
      \vdots & \ddots  & \ddots & \vdots \\
      B(\beta_{\kappa(\eta)}) & A(\beta_{\kappa(\eta)}) B(\beta_{\kappa(\eta)}) & \ldots & A^{n\kappa(\eta)-1}(\beta_{\kappa(\eta)})B(\beta_{\kappa(\eta)})
    \end{bmatrix}
  \end{aligned}
\]
to be the ensemble controllability matrix associated with the R-ensemble in \cref{equ: diagonal.linear.ensemble.general.eigenvalue}. Therefore, as a result of \cref{thm: UEC.general.complex.module.self.adjoint}, the R-ensemble in \cref{equ: diagonal.linear.ensemble.general.eigenvalue} is uniformly ensemble controllable if and only if its ensemble controllability matrix $G(\eta)$ has full row-rank for all $\eta\in \mathcal{K}$.

Next, we provide three examples to demonstrate the efficacy of \cref{thm: UEC.general.complex.module.self.adjoint}.

\begin{ex}
	Consider the RR-ensemble indexed by $\beta\in K$, given by
  \[
    \begin{aligned}
      \frac{\mathrm{d}}{\mathrm{d}t} X(t, \beta) &=
      \begin{bmatrix}
        \lambda_1(\beta) \\ & \ddots \\ & & \lambda_r(\beta)
      \end{bmatrix}X(t, \beta) + \sum_{j=1}^m b_j(\beta)u_j(t), \\
      &:= A(\beta)X(t, \beta) + B(\beta)u(t),
    \end{aligned}
  \]
	where $\lambda_j :K \to \mathbb{R}$, $j\in [1:r]$ are finite-to-one continuous functions over $K$. Then, its reachable set is given by $\mathcal{L} = \mathrm{span}\:\{ A^k b_j, j\in[1:m], k \in \mathbb{N} \}$. Since in this case, $A^{\ast} = A$, we have $\mathcal{L}$ is closed under $\varphi$, so that \cref{thm: UEC.general.complex.module.self.adjoint} applies. Therefore, the above linear ensemble is uniformly ensemble controllable on $C(K, \mathbb{R}^r)$ if and only if the induced system given by
	\[
		\frac{\mathrm{d}}{\mathrm{d}t} Y(t, \eta) = \mathcal{D}_A (\gamma^{-1}(\eta)) Y(t, \eta) + \mathcal{V}_B (\gamma^{-1}(\eta))u(t) 
	\]
	is controllable on $C(K, \mathbb{R}^{r \kappa (\eta)})$ for all $\eta\in \mathcal{K}$, where $\gamma^{-1}(\eta)$ can be explicitly written as
  \[
    \gamma^{-1}(\eta) = \{ \beta \in K: \lambda_j(\beta) = \eta \text{ for some } j\in [1:r]\}.
  \]
  As proved in \cref{lem: equivalence.to.Li20}, this sufficient and necessary condition is equivalent to the condition proposed in Theorem~3.7 in \cite{li2019separating}. \QEDA
\end{ex}

\begin{ex}
	Consider the RC-ensemble indexed by $\beta\in K$, which eigenvalues are all purely complex, given by
  \[
    \begin{aligned}
      \frac{\mathrm{d}}{\mathrm{d}t} X(t, \beta) &= \begin{bmatrix}
        G_1(\beta) \\ & \ddots \\ & & G_l(\beta)
        \end{bmatrix}X(t, \beta) + \sum_{j=1}^m b_j(\beta)u_j(t),\\
        &:= A(\beta)X(t, \beta) + B(\beta)u(t),
    \end{aligned}
  \]
	where $G_j(\beta) = \begin{bsmallmatrix} 0 & -\omega_j(\beta) \\ \omega_j(\beta) & 0	\end{bsmallmatrix}$ with $\omega_j: K \to \mathbb{R}$, $j = 1, \dots , l$, are finite-to-one continuous functions over $K$. In this case, $A^{\ast} = -A$ so that for any $f\in \mathcal{L}$, we have $A^*f = -Af\in \mathcal{L}$. Hence, \cref{thm: UEC.general.complex.module.self.adjoint} applies. Therefore, the above linear ensemble is uniformly ensemble controllable on $C(K,\mathbb{R}^{2l})$ if and only if the induced system given by
  \[
    \frac{\mathrm{d}}{\mathrm{d}t} Y(t, \eta) = \mathcal{D}_A (\gamma^{-1}(\eta)) Y(t, \eta) + \mathcal{V}_B (\gamma^{-1}(\eta))u(t) 
  \]
  is controllable on $C(K, \mathbb{R}^{2 \kappa (\eta)l})$ for all $\eta\in \mathcal{K}$, where $\gamma^{-1}(\eta)$ can be explicitly written as 
  \[
    \gamma^{-1}(\eta) = \{ \beta \in K: i\omega_j(\beta) = \eta \text{ or } -\eta \text{ for some } j\in [1:l]\}.
  \]
\end{ex}

\begin{ex}
	Consider the RC-ensemble indexed by $\beta\in K$, given by
	\[
		\frac{\mathrm{d}}{\mathrm{d}t} X(t, \beta) =
    \begin{bmatrix}
		  \sqrt{3} & -3 \\ 3 & \sqrt{3}
		\end{bmatrix} \phi(\beta) X(t, \beta) +
    \begin{bmatrix}
	    1 & -\frac{1}{2} & -\frac{1}{2} \\
      0 & -\frac{\sqrt{3}}{2} & \frac{\sqrt{3}}{2}
	  \end{bmatrix} \psi(\beta)u(t),
	\]
  where $\phi(\cdot ), \psi(\cdot ) \in C(K, \mathbb{R})$ with $\phi(\cdot )$ being finite-to-one. Due to its complicated dynamics, it is difficult to use existing approaches to analyze UEC of the above ensemble. We will show that the proposed algebraic framework in this work well-addresses the ensemble controllability of such RC-ensembles.  

  Let $A(\beta) = \begin{bsmallmatrix} \sqrt{3} & -3 \\ 3 & \sqrt{3}	\end{bsmallmatrix} \phi(\beta)$,
  $b_1(\beta) = \begin{bsmallmatrix} 1 \\ 0	\end{bsmallmatrix}\psi(\beta)$, $b_2(\beta) = \begin{bsmallmatrix}	-\frac{1}{2} \\ -\frac{\sqrt{3}}{2}	\end{bsmallmatrix} \psi(\beta)$,
  and $b_3(\beta) = \begin{bsmallmatrix}	-\frac{1}{2} \\ \frac{\sqrt{3}}{2} \end{bsmallmatrix} \psi(\beta)$.
  In this case, we have $A^*(\beta) = \begin{bsmallmatrix}	\sqrt{3} & 3 \\ -3 & \sqrt{3}	\end{bsmallmatrix} \phi(\beta)$. Then, it holds that
	\[
		A^*b_1 = \begin{bmatrix} \sqrt{3} & 3 \\ -3 & \sqrt{3} \end{bmatrix} \begin{bmatrix} 1 \\ 0 \end{bmatrix}\phi\psi
    = \begin{bmatrix} \sqrt{3} \\ -3 \end{bmatrix} \phi\psi
    = \begin{bmatrix} \sqrt{3} & -3 \\ 3 & \sqrt{3} \end{bmatrix} \begin{bmatrix}
		-\frac{1}{2} \\ -\frac{\sqrt{3}}{2}	\end{bmatrix}\phi\psi = Ab_2.
	\]
	Similarly, we can verify that $A^*b_2 = Ab_3$ and $A^*b_3 = Ab_1$ so that the reachable set $\mathcal{L}$ is closed under $\varphi$. Hence, \cref{thm: UEC.general.complex.module.self.adjoint} applies. Therefore, the above linear ensemble is uniformly ensemble controllable on $C(K, \mathbb{R}^2)$ if and only if the induced system given by
  \[
    \frac{\mathrm{d}}{\mathrm{d}t} Y(t, \eta) = \mathcal{D}_A (\gamma^{-1}(\eta)) Y(t, \eta) + \mathcal{V}_B (\gamma^{-1}(\eta))u(t) 
  \]
  is controllable on $C(K, \mathbb{R}^{2 \kappa (\eta)})$ for all $\eta\in \mathcal{K}$, where $\gamma^{-1}(\eta)$ can be explicitly written as
  \[
    \gamma^{-1}(\eta) = \{ \beta \in K: (\sqrt{3} + 3i) \phi(\beta) = \eta \text{ or } (\sqrt{3} - 3i) \phi(\beta) = \eta\}.
  \]
\end{ex}

Before concluding this work, it is worthwhile to mention that \cref{thm: UEC.general.complex.module.self.adjoint} requires the reachable set to be closed under the operation $\varphi$ defined in \cref{equ: UEC.general.complex.module.self.adjoint.varphi} so that the reachable set of the RC-ensemble can be written as a module of self-adjoint algebra. In this case, by checking whether continuous functions are localizable, the denseness of reachable set of a linear ensemble is fully determined by the separating properties of its control vector fields. However, as shown in the following example, when the reachable set cannot be written as a module of self-adjoint algebra, a function may not lie in the closure of the reachable set even if it is localizable by the reachable set. Hence, Stone-Weierstrass-type theorems may not be applicable to ensemble controllability analysis of all R-ensembles due to the nature of RC-ensembles. A sufficient and necessary condition for UEC of arbitrary RC-ensembles is beyond the scope of this work and requires further investigation.

\begin{ex}\label{ex: localizability.not.work.for.all.RC.ensemble}
  Consider the RC-ensemble indexed by $\beta\in K = [0, 1]$, given by
  \begin{equation}
      \frac{\mathrm{d}}{\mathrm{d}t} \begin{bmatrix}
      x_1 (t, \beta) \\ x_2(t, \beta)
      \end{bmatrix} = \begin{bmatrix}
      \cos(2\pi \beta) & -\sin(2\pi \beta) \\ \sin(2\pi \beta) & \cos(2\pi \beta)
      \end{bmatrix} \begin{bmatrix}
      x_1 (t, \beta) \\ x_2(t, \beta)
      \end{bmatrix} + \begin{bmatrix}
      u \\ v
      \end{bmatrix}, \label{equ: ex.non.self.adjoint.1}
  \end{equation}
  where $u, v:[0, T]\to \mathbb{R}$ are piecewise constant control signals. Let $z = x_1+ix_2$. Then, this linear ensemble is uniformly ensemble controllable on $C(K, \mathbb{R}^2)$ if and only if the C-ensemble given by 
  \[
    \frac{\mathrm{d}}{\mathrm{d}t} z(t, \beta) = (\cos(2\pi \beta) +i\sin(2\pi \beta))z(t, \beta) + w(t)= e^{2\pi i\beta} z(t, \beta) + w(t),
  \]
  is uniformly ensemble controllable on $C(K, \mathbb{C})$. We denote $a:K \to \mathbb{C}, a(\beta) = e^{2\pi i\beta}$. Then, the reachable set of the above C-ensemble, given by $\mathcal{A} = \mathrm{span}_{\mathbb{C}}\{a^k, k \in \mathbb{N}\}$ is a subalgebra of $C(K, \mathbb{C})$, and hence, is a complex-valued module of itself. We observe that $0\equiv 1 \pmod{\mathcal{A}}$ and $[\beta]_{\mathcal{A}} = \{\beta\}$ for all $\beta\in (0, 1)$. This is because if $\beta_1, \beta_2\in (0, 1)$ and $\beta_1 \equiv  \beta_2 \pmod{\mathcal{A}}$, then $a^k(\beta_1) = a^k(\beta_2)$ so that $e^{2\pi i k(\beta_1-\beta_2)} = 1$ holds for all $k \in \mathbb{N}$, which implies $\beta_1 = \beta_2$. Now, we consider the function $g:K \to \mathbb{C}$, $g(\beta) = e^{-2\pi i\beta}$. Then, for any $\beta_0\in [0, 1]$, the constant function $h_{[\beta_0]_{\mathcal{A}}}: [\beta_0]_{\mathcal{A}} \to \mathbb{C}$, $x\mapsto e^{-2\pi i\beta_0}$ satisfies $g\vert_{[\beta_0]_{\mathcal{A}}} = h_{[\beta_0]_{\mathcal{A}}}$. Hence, $g$ is localizable by $\mathcal{A}$. However, it is a well-known fact from Fourier analysis that $g\not\in \bar{\mathcal{A}}$. Therefore, we cannot examine UEC of such an RC-ensemble by enumerating all functions that are localizable by its reachable set, as we did in the proof of \cref{thm: UEC.general.complex.module.self.adjoint}. \QEDA
\end{ex}

\section{Conclusion}
\label{sec: conclusion}
In this paper, we have proposed a new algebraic framework to examine UEC of RC-ensembles leveraging Stone-Weierstrass theorem for modules. By checking the rank condition of the ensemble controllability matrix, we provide a sufficient and necessary condition for UEC of a large class of R-ensembles. We have demonstrated the proposed framework is widely applicable to various R-ensembles and well-encompasses existing results; however, it is still subject to some limitations due to the nature of RC-ensembles. A sufficient and necessary condition for UEC of general $n$-dimensional RC-ensembles worths further investigation.

\appendix
\section{~}

\begin{ex}
  \label{ex: motivating example 2}
  Let us consider a two-dimensional RC-ensemble, indexed by a parameter $\beta\in K \subset \mathbb{R}$, given by
  \begin{equation}
    \frac{\mathrm{d}}{\mathrm{d}t} \begin{bmatrix}
    x_1(t, \beta) \\ x_2(t, \beta)
    \end{bmatrix} = \begin{bmatrix}
    \alpha(\beta) & -\omega(\beta) \\ \omega(\beta) & \alpha(\beta)
    \end{bmatrix} \begin{bmatrix}
    x_1(t, \beta) \\ x_2(t, \beta)
    \end{bmatrix} + \begin{bmatrix}
    1 & 0 \\ 0 & \beta
    \end{bmatrix} \begin{bmatrix}
    u \\ v
    \end{bmatrix}, \label{equ: motivating example 2-1}
  \end{equation}
where $x_1(t, \cdot), x_2(t, \cdot), \alpha(\cdot ), \omega(\cdot )\in C(K, \mathbb{R})$ and $u, v :[0, T]\to \mathbb{R}$ are piecewise constant. 

By applying a coordinate transformation $z(t, \beta) = x_1(t, \beta) + i x_2(t, \beta)$, the dynamics of $z(t, \beta)$ yields (with $t$ and $\beta$ omitted for ease of exposition):\
\begin{equation}
  \label{equ: motivating example 2-2}
  \begin{aligned}
    \frac{\mathrm{d}z}{\mathrm{d}t} &= \frac{\mathrm{d}x_1}{\mathrm{d}t}+i \frac{\mathrm{d}x_2}{\mathrm{d}t} = \alpha x_1 - \omega x_2 + i(\omega x_1 + \alpha x_2) + u + i \beta v \\
    &= (\alpha + i \omega) z + u + i \beta v.
  \end{aligned}
\end{equation}
Then the ensemble in \cref{equ: motivating example 2-1} is uniformly ensemble controllable on $C(K, \mathbb{R}^2)$ if and only if the ensemble in \cref{equ: motivating example 2-2} is uniformly ensemble controllable on $C(K, \mathbb{C})$. Hence we only need to analyze the UEC of the complex-valued ensemble system in \cref{equ: motivating example 2-2}. However, due to the parameter variation in the control vector field, the ensemble in \cref{equ: motivating example 2-2} is not equivalent to any C-ensemble. This is because we cannot find a complex-valued vector fields $\eta(\beta)$ such that the two sets $S_1, S_2 \subset C(K, \mathbb{C})$ defined by
\[
  S_1 = \{u+ i\beta v: u, v\in\mathbb{R}\}, \qquad S_2 = \{\eta(\beta)w : w\in \mathbb{C}\}.
\]
are equal. Hence the RC-ensemble in \cref{equ: motivating example 2-1} cannot be turned into a C-ensemble with both vector fields and control signals being complex-valued.\QEDA
\end{ex}

\begin{lemma}
  \label{lem: equivalence.to.Li20}
  Consider the RR-ensemble given by
  \[
    \frac{\mathrm{d}}{\mathrm{d}t} X(t, \beta) =
    \begin{bmatrix}
      \lambda_1(\beta) \\ & \ddots \\ & & \lambda_n(\beta)
    \end{bmatrix}X(t, \beta) +
    \begin{bmatrix}
      b_1^T(\beta) \\ \vdots \\ b_n^T(\beta)
    \end{bmatrix}u(t),
  \]
  with $\lambda_j:K \to \mathbb{R}$, $j\in [1:n]$ are continuous functions. Then, the subsystem indexed by $\gamma^{-1}(\eta)$ given by
  \begin{equation}
    \label{equ: condition.this.work}
    \frac{\mathrm{d}}{\mathrm{d}t} Y(t, \eta) = \mathcal{D}_A(\gamma^{-1}(\eta))Y(t, \eta)p + \mathcal{V}_B(\gamma^{-1}(\eta))u(t),
  \end{equation}
  is controllable on $C(K, \mathbb{R}^{n\kappa (\eta)})$ for all $\eta\in \mathcal{K}$ if and only if the system indexed by $\eta_1, \ldots , \eta_n$ given by
  \begin{equation}
    \label{equ: condition.in.Li20}
    \frac{\mathrm{d}}{\mathrm{d}t} \begin{bmatrix} Z_1(t, \eta_1) \\ \vdots \\ Z_n(t, \eta_n) \end{bmatrix} =
    \begin{bmatrix}
      \eta_{1} I_{\kappa_1(\eta_1)} \\ & \ddots \\ & & \eta_n I_{\kappa_n(\eta_n)}
    \end{bmatrix}
    \begin{bmatrix}
      Z_1(t, \eta_1) \\ \vdots \\ Z_n(t, \eta_n)
    \end{bmatrix} + 
    \begin{bmatrix}
      D_1(\eta_1) \\ \vdots \\ D_n(\eta_n)
    \end{bmatrix},
  \end{equation}
  is ensemble controllable on $C(K, \mathbb{R}^N)$ with $N = \sum_{j=1}^n \kappa_j(\eta_j)$ for all $(\eta_1,\ldots , \eta_n)\in \lambda_1(K)\times \cdots \times \lambda_n(K)$, where $D_j(\eta_j)$ is $\eta_j$-separation matrix of the system
  \[
    \frac{\mathrm{d}}{\mathrm{d}t} z(t, s) = \lambda_j(s) z(t, s) + b_j(s) u(t).
  \]
\end{lemma}

\begin{proof}
  Without loss of generality, we prove the case when $n = 2$. The proof can be generalized straightforwardly to the cases with $n>2$.

  (Necessity): Suppose $\lambda_1(\beta_1) = \eta_1$ and $\lambda_2(\beta_2) = \eta_2$ with $\eta_1 \neq \eta_2$. Without loss of generality, we assume $\eta_1\not\in K_2\setminus K_1$ and $\eta_2\not \in K_1\setminus K_2$ and $\kappa_1(\eta_1) = \kappa_2(\eta_2) = 1$. If the system in \cref{equ: condition.this.work} is controllable on $C(K, \mathbb{R}^{n\kappa (\eta)})$ for all $\eta\in \mathcal{K}$, then the system
  \[\begin{aligned}
    \frac{\mathrm{d}}{\mathrm{d}t}
    \begin{bmatrix}
      Y_1(t, \eta_1) \\ Y_2 (t, \eta_2)
    \end{bmatrix} &=
    \begin{bmatrix}
      \mathcal{D}_A(\gamma^{-1}(\eta_1)) \\
        & \mathcal{D}_A(\gamma^{-1}(\eta_2))
    \end{bmatrix}
    \begin{bmatrix}
      Y_1(t, \eta_1) \\
      Y_2 (t, \eta_2)
    \end{bmatrix} +
    \begin{bmatrix}
      \mathcal{V}_B(\gamma^{-1}(\eta_1)) \\
      \mathcal{V}_B(\gamma^{-1}(\eta_2))
    \end{bmatrix}u(t) \\
    &=\begin{bmatrix}
      \lambda_1(\beta_1) \\
        & \lambda_2(\beta_1) \\
        & & \lambda_1(\beta_2) \\
        & & & \lambda_2(\beta_2)
    \end{bmatrix}
    \begin{bmatrix}
      Y_1(t, \eta_1) \\ Y_2 (t, \eta_2)
    \end{bmatrix}+
    \begin{bmatrix}
      b^T_1(\beta_1) \\ b^T_2(\beta_1) \\ b^T_1(\beta_2) \\  b^T_2(\beta_2)
    \end{bmatrix},
  \end{aligned}\]
  is controllable on $\mathbb{R}^4$. Then, we have $\mathrm{rank}\{
  \begin{bsmallmatrix}
    \lambda^k_1(\beta_1)b^T_1(\beta_1) \\
    \lambda^k_2(\beta_1) b^T_2(\beta_1) \\
    \lambda^k_1(\beta_2) b^T_1(\beta_2) \\
    \lambda_2^k(\beta_2)  b^T_2(\beta_2)
  \end{bsmallmatrix}, k \in \mathbb{N}\}=4$ so that
  \[
    \mathrm{rank} \{\begin{bmatrix}
      \lambda^k_1(\beta_1)b^T_1(\beta_1) \\ \lambda^k_2(\beta_2) b^T_2(\beta_2) 
    \end{bmatrix}, k \in \mathbb{N} \} = \mathrm{rank} \{\begin{bmatrix}
      \eta_1^kb^T_1(\lambda^{-1}(\eta_1)) \\ \eta_2^k b^T_2(\lambda^{-1}(\eta_2)) 
    \end{bmatrix}, k \in \mathbb{N} \} = 2,
  \]
  which implies that the system 
  \[
    \frac{\mathrm{d}}{\mathrm{d}t}
    \begin{bmatrix}
      z_1(t, \eta_1) \\ z_2(t, \eta_2)
    \end{bmatrix} =
    \begin{bmatrix}
      \eta_1 \\ & \eta_2
    \end{bmatrix}
    \begin{bmatrix}
      z_1(t, \eta_1) \\ z_2(t, \eta_2)
    \end{bmatrix} +
    \begin{bmatrix}
      b^T_1(\lambda^{-1}(\eta_1)) \\ b^T_2(\lambda^{-1}(\eta_2)) 
    \end{bmatrix} u(t)
  \]
  is controllable on $\mathbb{R}^2$. The proof for the cases with $\eta_1 = \eta_2$ or $\kappa_j(\eta_j) \geq 2$, $j = 1,2$ follows the similar procedures above and therefore is omitted.

  (Sufficiency): We observe that given distinct real numbers $a_1, \ldots , a_n$, and non-zero real numbers $c_1, \ldots , c_n$, it always hold that $\rank \{c, Ac, \ldots , A^{n-1}c\} = n$ where $A=
  \begin{bsmallmatrix}
    a_1 \\ & \ddots \\ & & a_n
  \end{bsmallmatrix}$ and
  $c=\begin{bsmallmatrix}
    c_1 \\ \vdots \\ c_n
  \end{bsmallmatrix}$. Leveraging this fact, we argue case by case that if the system in \cref{equ: condition.this.work} is not controllable for some $\eta\in \mathcal{K}$, then the system in \cref{equ: condition.in.Li20} is not controllable for some $(\eta_1, \ldots , \eta_n)\in \lambda_1(K)\times \cdots \times \lambda_n(K)$. 

  Without loss of generality, we consider the case where $\lambda_j$, $j = 1, 2$ has at most $2$ non-injective branches. Let $\gamma \in \mathcal{K}$ with $\gamma^{-1}(\eta) = \{\beta_1, \beta_2\}$. Then, either $\lambda_1(\beta_1) = \lambda_1(\beta_2) = \eta$ or $\lambda_1(\beta_1) = \lambda_2(\beta_2) = \eta$. We further assume that $\lambda_2(\beta_1)\neq \lambda_2(\beta_2) \neq \eta$ since the proof for the case with $\lambda_2(\beta_1) = \eta$ or $\lambda_2(\beta_2) = \eta$ follows the same logic below.
  \begin{enumerate}[label=(\roman*), leftmargin= 2em]
    \item If $\lambda_1(\beta_1) = \lambda_1(\beta_2) = \eta$, then the system in \cref{equ: condition.this.work} boils down to the system
    \[
      \frac{\mathrm{d}}{\mathrm{d}t} X(t,\gamma^{-1}(\eta))=
      \begin{bmatrix}
        \eta \\ & \eta \\& & \lambda_2(\beta_1) \\ & & & \lambda_2(\beta_2)
      \end{bmatrix} X(t,\gamma^{-1}(\eta)) + \begin{bmatrix}
        b^T_1(\beta_1) \\ b^T_1(\beta_2) \\ b^T_2(\beta_1) \\ b^T_2(\beta_2)
          \end{bmatrix} u(t).
    \]
    If such a system is not controllable on $\mathbb{R}^4$, then one of the followings happens
    \begin{enumerate}[label=(\arabic*), leftmargin= 2em]
      \item $\rank\begin{bsmallmatrix} b^T_1(\beta_1) \\ b^T_1(\beta_2) \end{bsmallmatrix}<2$. Then, the system in \cref{equ: condition.in.Li20} indexed by $\eta_1 = \eta$ and arbitrary $\eta_2$ is uncontrollable.
      \item $b_2^T(\beta_1) = 0$. Then, the system in \cref{equ: condition.in.Li20} indexed by $\eta_1=\eta, \eta_2 = \lambda_2(\beta_1)$ is uncontrollable.
      \item $b_2^T(\beta_2) = 0$. Then, the system in \cref{equ: condition.in.Li20} indexed by $\eta_1=\eta, \eta_2 = \lambda_2(\beta_2)$ is uncontrollable.
    \end{enumerate}        
    \item If $\lambda_1(\beta_1) = \lambda_2(\beta_2) = \eta$, then the system in \cref{equ: condition.this.work} boils down to the system
    \[
      \frac{\mathrm{d}}{\mathrm{d}t} X(t, \gamma^{-1}(\eta))=
      \begin{bmatrix}
        \eta \\ & \eta \\& & \lambda_1(\beta_2) \\ & & & \lambda_2(\beta_1)
      \end{bmatrix} X(t, \gamma^{-1}(\eta)) + 
      \begin{bmatrix}
        b^T_1(\beta_1) \\ b^T_2(\beta_2) \\ b^T_1(\beta_2) \\ b^T_2(\beta_1)
      \end{bmatrix} u(t).
    \]
    If such a system is not controllable on $\mathbb{R}^4$, then one of the followings happens
    \begin{enumerate}[label=(\arabic*), leftmargin= 2em]
      \item $\rank \begin{bsmallmatrix} b^T_1(\beta_1) \\ b^T_2(\beta_2) \end{bsmallmatrix}<2$. Then, the system in \cref{equ: condition.in.Li20} indexed by $\eta_1 = \eta_2 = \eta$ is uncontrollable.
      \item $b_1^T(\beta_2) = 0$. Then, the system in \cref{equ: condition.in.Li20} indexed by $\eta_1=\eta$, $\eta_2 = \lambda_1(\beta_2)$ is uncontrollable.
      \item $b_2^T(\beta_1) = 0$. Then, the system in \cref{equ: condition.in.Li20} indexed by $\eta_1=\eta$, $\eta_2 = \lambda_2(\beta_1)$ is uncontrollable.
    \end{enumerate}
  \end{enumerate}
\end{proof}

\section*{Acknowledgments}
  We would like to acknowledge the assistance of volunteers in putting together this example manuscript and supplement.

\bibliographystyle{siamplain}
\bibliography{UEC_complex_eigenvalues}
\end{document}